\documentclass[11pt]{article}
\usepackage[utf8]{inputenc}
\usepackage[margin=1in]{geometry}

\usepackage{amsmath, amsthm, thmtools}

\usepackage{fourier}
\DeclareMathAlphabet{\mathcal}{OMS}{zplm}{m}{n}

\usepackage[dvipsnames]{xcolor}
\usepackage[pagebackref]{hyperref}
\hypersetup{
    colorlinks=true,      
    linkcolor=Green,       
    citecolor=NavyBlue,      
    filecolor=Plum,    
    urlcolor=Plum         
}
\usepackage{mathtools, amsfonts, amssymb, mathrsfs}
\usepackage{enumerate,enumitem}
\usepackage{subcaption}
\usepackage{soul}
\usepackage{algorithm, algpseudocode, algorithmicx}
\usepackage{comment}
\usepackage{xspace}
\usepackage{mleftright}

\newcommand{\mytitle}{A new analysis of the randomly pivoted Cholesky algorithm}

\renewcommand*{\backref}[1]{}
\renewcommand*{\backrefalt}[4]{%
	\ifcase #1 %
	(No citations.)
	\or
	(Cited page~#2.)
	\else
	(Cited pages~#2.)
	\fi
}

\usepackage[most]{tcolorbox}

\newcommand{\real}{\mathbb{R}}
\newcommand{\complex}{\mathbb{C}}

\DeclareMathOperator{\tr}{tr}
\DeclareMathOperator{\diag}{diag}

\newcommand{\mat}[1]{\boldsymbol{#1}}
\renewcommand{\vec}[1]{\boldsymbol{#1}}
\newcommand{\lowrank}[1]{\mleft\llbracket #1 \mright\rrbracket}
\newcommand{\norm}[1]{\mleft\| #1 \mright\|}

\newcommand{\Id}{\mathbf{I}}

\DeclareMathOperator{\expect}{\mathbb{E}}
\DeclareMathOperator{\prob}{\mathbb{P}}

\newcommand{\order}{\mathcal{O}}

\newcommand{\set}[1]{\mathsf{#1}}
\newcommand{\e}{\mathrm{e}}

\newcommand{\hatbold}[1]{\skew{4}\widehat{\smash{\boldsymbol{#1}}\mathstrut}}
\newcommand{\Ahat}{\smash{\hatbold{A}}}

\newcommand{\RPCholesky}{\textsc{RPCholesky}\xspace}
\renewcommand{\d}{\mathrm{d}}
\newcommand{\ridge}{\mu}
\newcommand{\otherdeff}[1]{\mathrm{d}_{\rm eff}(#1)}
\newcommand{\deff}{\otherdeff{\ridge}}
\newcommand{\dtail}{\mathrm{d}_{\rm tail}(\ridge)}

\usepackage[nameinlink,capitalise]{cleveref}
\usepackage{doi}

\makeatletter
\def\th@plain{%
  \thm@notefont{}
  \itshape 
}
\def\th@definition{%
  \thm@notefont{}
  \normalfont 
}
\makeatother

\crefname{equation}{}{}
\crefname{section}{section}{sections}
\crefname{appendix}{appendix}{appendices}
\newcommand*{\email}[1]{\href{mailto:#1}{\nolinkurl{#1}} } 

\declaretheorem[name=Theorem,numberwithin=section]{theorem}

\declaretheorem[name=Proposition,numberlike=theorem]{proposition}

\declaretheorem[name=Known Result,numberlike=theorem]{known}

\declaretheorem[name=Lemma,numberlike=theorem]{lemma}
\declaretheorem[name=Corollary,numberlike=theorem]{corollary}

\theoremstyle{remark}

\theoremstyle{definition}

\usepackage{mdframed}
\newmdtheoremenv{question}{Question}

\numberwithin{equation}{section}

\usepackage{titlesec}

\titleformat{\subsection}[runin]
  {\normalfont\normalsize\scshape\bfseries}{\thesubsection}{0.5em}{}[\textbf{.}]

  \titleformat{\subsubsection}[runin]
  {\normalfont\normalsize\scshape}{\textit{\thesubsubsection}}{0.5em}{}[.]

  \titleformat{\paragraph}[runin]
  {\normalfont\normalsize\bfseries\itshape}{\textit{\theparagraph}}{0.5em}{}[]

\title{\mytitle}
\author{Ethan N. W. Epperly\thanks{Department of Mathematics, University of California Berkeley (\href{mailto:eepperly@berkeley.edu}{eepperly@berkeley.edu}, \url{https://ethanepperly.com})}}
\date{\today}

\begin{document}

\maketitle

\begin{abstract}
  The randomly pivoted Cholesky algorithm is one of the leading methods for computing a low-rank approximation to a large positive-semidefinite matrix.
However, while it consistently achieves accuracy comparable to or better than competing methods of its type in experiments, its theoretical analysis lags somewhat behind other methods.
This paper closes this gap, proving that randomly pivoted Cholesky produces an approximation with expected error within a $1+\varepsilon$ factor of the optimal rank-$r$ approximation in $\order(r/\varepsilon + r\sqrt{\log r})$ steps.
This result nearly matches the optimal complexity $\Theta(r/\varepsilon)$ for any low-rank approximation method based on a partial Cholesky decomposition (also known as a column Nystr\"om approximation).
The paper also presents bounds on the randomly pivoted Cholesky trace and spectral-norm errors that hold with high probability.
The mathematical argument is largely due to GPT 5.6-Sol (Pro), with some refinements by the author.
\end{abstract}

\section{Introduction}

In recent years, the randomly pivoted Cholesky (\RPCholesky) algorithm \cite{MW19,CETW25} has established itself as one of the most effective methods for computing a low-rank approximation to a positive semidefinite (psd) matrix, consistently achieving better accuracy or lower computational cost than competitor methods \cite{DEF+23,CETW25,ETW25a}.
However, while the empirical performance has been strong, the error analysis of \RPCholesky has lagged somewhat behind other methods, leaving open the possibility that this algorithm could have a ``hidden weakness'' not visible in the current numerical work.
This paper resolves this question by providing error bounds for \RPCholesky that nearly achieve the best-possible error scaling for the psd low-rank approximation problem.
The mathematical argument is largely due to GPT 5.6-Sol (Pro), with some refinements by the author.

\subsection{\RPCholesky}

Let $\mat{A} \in \complex^{N\times N}$ be a psd matrix, defined as a matrix with nonnegative eigenvalues that is equal to its adjoint $\mat{A}^*$.
The \RPCholesky algorithm builds an approximation iteratively, beginning from the trivial approximation $\Ahat^{(0)} \coloneqq \mat{0}$.
At step $i$, the algorithm draws a random pivot index
\begin{equation*}
	\prob \bigl\{ s_i = t \bigr\} = \frac{\mat{A}(t,t) - \Ahat^{(i-1)}(t,t)}{\tr\bigl(\mat{A} - \Ahat^{(i-1)}\bigr)}
\end{equation*}
using the diagonal of the current residual $\mat{A} - \Ahat^{(i-1)}$ as sampling weights.
Then, it updates the approximation using the selected pivot column
\begin{equation} \label{eq:cholesky-update}
	\Ahat^{(i)} \coloneqq \Ahat^{(i-1)} + \vec{v}_i^{\vphantom{*}}\vec{v}_i^*/\vec{v}_i^{\vphantom{*}}\bigl(s_i^{\vphantom{*}}\bigr) \quad \text{for } \vec{v}_i^{\vphantom{*}} \coloneqq \mat{A}\bigl(:,s_i^{\vphantom{*}}\bigr) - \Ahat^{(i-1)}\bigl(:,s_i^{\vphantom{*}}\bigr).
\end{equation}

Each step of the \RPCholesky algorithm requires interacting with the diagonal and a single column $\mat{A}(:,s_i)$ of the matrix.
As such, the algorithm produces a rank-$k$ approximation after reading only $(k+1)N$ matrix entries, and, with appropriate implementation \cite{CETW25,ETW25a}, it requires only $\order(k^2N)$ arithmetic operations.
Treating the parameter $k$ as fixed, the algorithm is \emph{sublinear}, requiring only $\order(N)$ operations to produce an approximation to all $N^2$ entries of the matrix $\mat{A}$.
This computational profile makes the \RPCholesky method attractive for kernel matrix and Gaussian process computations \cite{SS02,RW05}, which involve large psd matrices $\mat{A} \in \complex^{N\times N}$ whose entries $\mat{A}(i,j) = \kappa(\vec{x}_i,\vec{x}_j)$ are specified by evaluating a kernel function $\kappa$ at pairs of elements from a data set $\{\vec{x}_i\}$.
The available numerical evidence suggests that \RPCholesky is among the fastest and most accurate methods for constructing a low-rank approximation to a psd matrix \cite{DEF+23,CETW25,ETW25a,Epp25a}, and the method has been shown to be useful for various tasks such as preconditioning \cite{DEF+23,LR25,BMS26,KW26}, quadrature \cite{EM23a,Epp25a}, matrix function approximation \cite{PMM25}, understanding chemical dynamics \cite{AJSW24a}, model reduction \cite{Oeh26}, and approximating the attention mechanism \cite{SM26}.

\subsection{Analysis: Trace error}

The \RPCholesky algorithm has proven effective and reliable in practice, but the existing analysis falls somewhat short of showing this.
The best-known bound is as follows.

\begin{known}[\RPCholesky: Existing analysis] \label{known:rpcholesky}
	Let $r\ge 1$ and $\varepsilon \in (0,1)$.
	The \RPCholesky algorithm achieves a trace error competitive with the best rank-$r$ approximation
	\begin{equation} \label{eq:expectation}
		\expect\bigl[\tr\bigl(\mat{A} - \Ahat^{(k)}\bigr)\bigr] \le (1+\varepsilon) \tr(\mat{A} - \lowrank{\mat{A}}_r).
	\end{equation}
	once the algorithm has been run for 
	\begin{equation*}
		k \ge \frac{r}{\varepsilon} + \min \left\{ r \log \left( \frac{1}{\varepsilon\eta}\right), \log(2) r^2 + r + r \log\left(\frac{1}{\varepsilon}\right) \right\} \text{ steps}.
	\end{equation*}
	Here, $\lowrank{\mat{A}}_r$ is the optimal rank-$r$ approximation to $\mat{A}$ in any unitarily invariant norm, and
	\begin{equation} \label{eq:cetw}
		\eta \coloneqq \tr(\mat{A} - \lowrank{\mat{A}}_r)/\tr(\mat{A}) \text{ is its relative error.}
	\end{equation}
\end{known}

This result shows that \RPCholesky produces an approximation that is comparable to the best rank-$r$ approximation when run for an appropriate number $k = k(r,\varepsilon)$ of steps.
It is known that any algorithm that produces a low-rank approximation by means of the partial Cholesky update \cref{eq:cholesky-update} requires $k \ge r/\varepsilon$ steps to satisfy the guarantee \cref{eq:expectation}, and the bound \cref{eq:expectation} is achieved with $k = r/\varepsilon + r - 1$ when the pivots $\{s_1,\ldots,s_k\}$ are drawn from a determinantal point process (DPP); see \cite[Thm.~6]{DM21a} and \cite[Fact~3.11]{Epp25a}.
Thus, when $\eta$ is large, \cref{known:rpcholesky} shows that \RPCholesky achieves error bounds that are nearly optimal within the class of partial Cholesky approximations (also called column Nystr\"om approximations \cite[sec.~19.2.1]{MT20b}).

\Cref{known:rpcholesky} looks less positive when we consider a matrix $\mat{A}$ with rapidly decaying eigenvalues.
For instance, if the eigenvalues of $\mat{A}$ decrease at a sufficiently rapid exponential rate, then this bound suggests it takes \RPCholesky roughly $0.7r^2$ steps to produce an approximation that is within a factor two of the best rank-$r$ approximation.
DPP sampling, by contrast, requires only $2r$ steps to meet this guarantee.
The analysis suggests a yawning chasm between these two methods. 
Yet, in every documented experiment, they always achieve similar levels of accuracy \cite{EM23a,CETW25}.

This paper closes this gap in the analysis of the \RPCholesky algorithm, showing that the method requires at most roughly $2r\sqrt{\log r}$ steps to produce an approximation comparable to the best rank-$r$ approximation.
In particular, we have the following theorem and corollary.

\begin{theorem}[\RPCholesky: New bound] \label{thm:main}
	Let $r\ge 1$ and $\varepsilon > 0$, and let $\Ahat^{(k)}$ denote the output of $k\ge r$ steps of \RPCholesky on $\mat{A}$.
	Then
	\begin{equation*}
		\prob \big\{ \tr\bigl(\mat{A} - \Ahat^{(k)}\bigr) \ge (1+\varepsilon) \tr(\mat{A} - \lowrank{\mat{A}}_r) \big\} \le \frac{k!}{(k-r)!} \cdot (1+\varepsilon)^{-(k-r)}.
	\end{equation*}
\end{theorem}

\begin{corollary}[RPCholesky: Trace error] \label{cor:complexity}
	Let $r\ge 1$ and $\varepsilon,\delta \in (0,1)$, and let $\Ahat^{(k)}$ denote the output of $k$ steps of \RPCholesky on $\mat{A}$.
	Then 
	\begin{subequations} \label{eq:result}
		\begin{align}
			&\expect\bigl[\tr\bigl(\mat{A} - \Ahat^{(k)}\bigr)\bigr] \le (1+\varepsilon) \tr(\mat{A} - \lowrank{\mat{A}}_r) &&\text{when } k\ge \frac{r}{\varepsilon} + 2r\sqrt{\log r} \textcolor{gray}{+ r \log\left(\frac{1}{\varepsilon}\right) + 2.3r}, \label{eq:expectation-result} \\
			&\prob \bigl\{ \tr\bigl(\mat{A} - \Ahat^{(k)}\bigr) > (1+\varepsilon) \tr(\mat{A} - \lowrank{\mat{A}}_r) \bigr\} \le \delta &&\text{when } k\ge k_{\rm hp}(r,\varepsilon,\delta), \label{eq:whp-result}
		\end{align}
	\end{subequations}
	where $k_{\rm hp}$ is an explicit function defined below in \cref{eq:k_hp} that satisfies the following bound and asymptotics:
	\begin{align*}
		k_{\rm hp}(r,\varepsilon,\delta) \le \frac{4r}{\varepsilon} \log\left( \frac{4r}{\varepsilon} \right) + \frac{4}{\varepsilon} \log \left( \frac{1}{\delta} \right) \quad \text{and} \quad k_{\rm hp}(r,\varepsilon,\delta) \sim \frac{r}{\varepsilon} \log \left( \frac{r}{\varepsilon} \right) + \frac{1}{\varepsilon} \log \left(\frac{1}{\delta}\right) \text{ as } r\uparrow \infty \text{ and } \varepsilon,\delta \downarrow 0.
	\end{align*}
\end{corollary}

The main conclusion is that it takes \RPCholesky roughly at most $r/\varepsilon + 2r\sqrt{\log r}$ steps to produce an approximation within a $(1+\varepsilon)$ factor of the optimal rank-$r$ approximation, in expectation.
The lower-order terms, which are grayed out to distinguish them from the leading terms, are small as well.
Robert Webber suggested the following simpler variant of \cref{eq:expectation-result}, which can also be deduced from the proof:
\begin{equation}
	\expect\bigl[\tr\bigl(\mat{A} - \Ahat^{(k)}\bigr)\bigr] \le (1+\varepsilon) \tr(\mat{A} - \lowrank{\mat{A}}_r) \quad \text{when } k\ge \frac{r}{\varepsilon} + 4r\sqrt{\log (\e r)} \textcolor{gray}{+ r \log\left(\frac{1}{\varepsilon}\right)}.
\end{equation}

It is natural to wonder whether the $\sqrt{\log r}$ factor in this analysis is necessary.
The efforts of the author, over many years, and GPT 5.6 (Pro), over several prompts, have failed to produce an example where \RPCholesky requires $\omega(r)$ steps to achieve the guarantee \cref{eq:expectation} with constant $\varepsilon$.
On that basis, one may conjecture that \RPCholesky requires only $\order(r/\varepsilon)$ steps to achieve the guarantee \cref{eq:expectation}.

\subsection{Discussion: Trace error}

To compare the new and old bounds, we evaluate on a simple test matrix $\mat{A}$, chosen to be a kernel matrix on a regular $80\times 80$ grid of points on the unit square $[0,1]^2$.
We evaluate with both the squared-exponential (Gaussian) and Mat\'ern-5/2 kernels, and we set the bandwidth to 0.3, which was manually tuned so that the eigenvalues decayed roughly as $0.8^i$ for the squared-exponential kernel.
We added a small shift $\mat{A} \gets \mat{A} + 10^{-14}\lambda_{\rm max}(\mat{A})\Id$ to avoid numerical issues.
We compare the new bound (\cref{cor:complexity}) to the old bound (\cref{known:rpcholesky}).
For each value of $k$ and each bound, we optimize to find the value of $r$ and $\varepsilon$ that most tightly controls the error.

\begin{figure}
	\centering
	\includegraphics[width=0.99\linewidth]{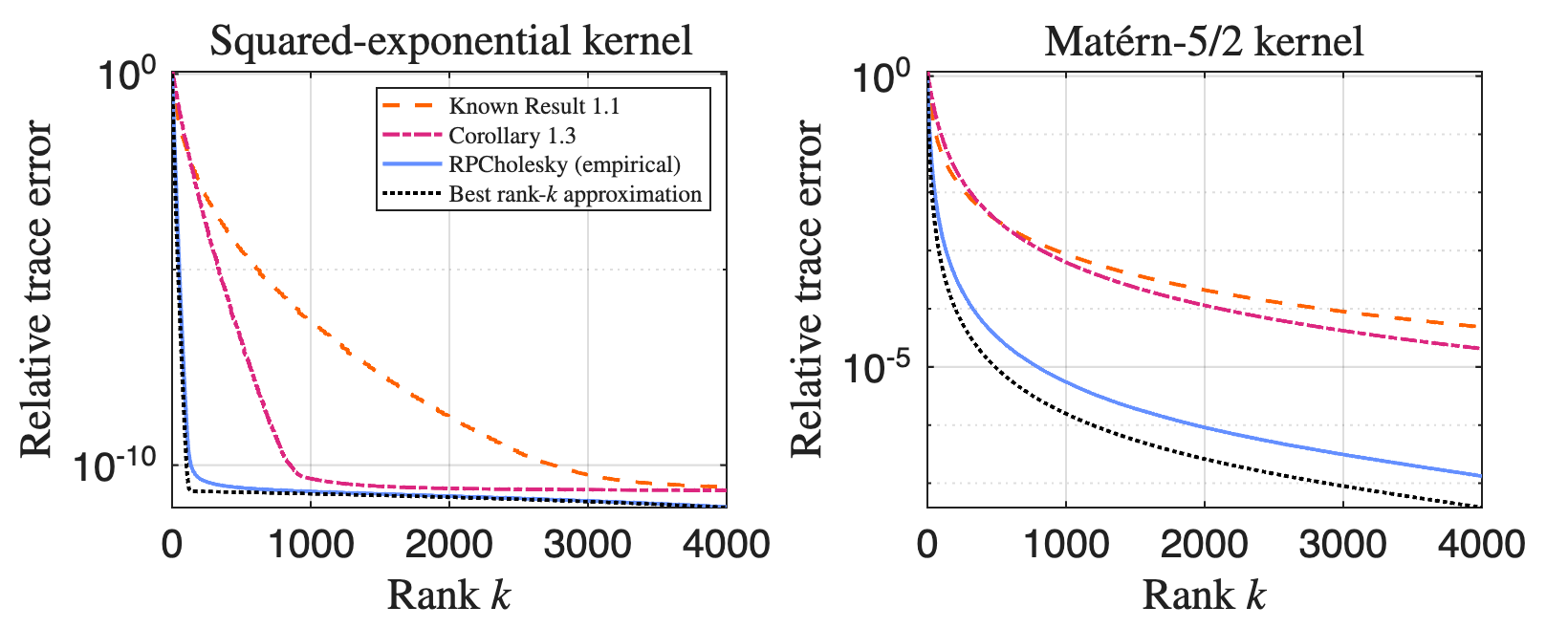}
	\caption{Relative trace error $\tr(\mat{A} - \Ahat) / \tr(\mat{A})$ of \RPCholesky (blue solid) as a function of the rank $k$ on two kernel matrices (described in text) and predictions by two error bounds: \cref{known:rpcholesky} (pink dash-dotted) and \cref{cor:complexity} (orange dashed).
		The error achieved by the best rank-$k$ approximation (black dotted) is shown for reference, and we plot the mean error for \RPCholesky over 10 trials.}
	\label{fig:trace}
\end{figure}

\Cref{fig:trace} shows the results.
For the squared-exponential kernel, the eigenvalues of $\mat{A}$ decrease at an exponential rate, and the trace error for \RPCholesky decreases exponentially as well.
\Cref{known:rpcholesky} fails to predict this behavior, instead predicting a root-exponential convergence profile of the form $\exp(-c\sqrt{k})$.
The new result (\cref{cor:complexity}) is better, correctly certifying that \RPCholesky converges exponentially.
While the new bound is qualitatively better, the bounds it produces are still far from quantitatively sharp.
For the Mat\'ern kernel, the new and old bounds are similar, with only slight improvements in \cref{cor:complexity} over \cref{known:rpcholesky}.

As \cref{fig:trace} demonstrates, none of the existing error bounds for \RPCholesky are quantitatively sharp.
Nonetheless, \cref{cor:complexity} provides a substantial improvement in some cases.

\subsection{Analysis: Spectral-norm error}

\Cref{cor:complexity} bounds the number of steps needed by \RPCholesky to produce an accurate approximation when measured using the trace error.
But for many purposes, we need stronger, spectral-norm guarantees.
The output $\Ahat$ of \RPCholesky is always bounded from above by $\mat{A}$ in the psd order: $\Ahat \preceq \mat{A}$.
We say that $\Ahat$ is a \emph{$\mu$-spectral approximation} to $\mat{A}$ if we have a matching upper bound, up to an additive factor:
\begin{equation} \label{eq:spectral}
	\Ahat \preceq \mat{A} \preceq \Ahat + \ridge \Id.
\end{equation}
In view of the relation $\Ahat \preceq \mat{A}$, the spectral approximation property \cref{eq:spectral} is equivalent to a bound $\norm{\mat{A} - \Ahat} \le \ridge$ on the error in the spectral norm.

To investigate the number of \RPCholesky steps needed to achieve the guarantee \cref{eq:spectral}, we introduce the \emph{$\ridge$-effective dimension}
\begin{equation*}
	\deff \coloneqq \tr \big[\mat{A} (\mat{A} + \ridge \Id)^{-1}\big] = \sum_{i=1}^N \frac{\lambda_i(\mat{A})}{\lambda_i(\mat{A}) + \ridge},
\end{equation*}
which is standard in the analysis of psd low-rank approximation algorithms \cite{MM17,RCCR18,FTU23}.
The $\mu$-effective dimension is a smoothed count of the number of eigenvalues larger than $\mu$, with large eigenvalues $\lambda_i(\mat{A}) \gg \ridge$ contributing nearly 1 to the sum and small eigenvalues $\lambda_i(\mat{A}) \ll \ridge$ contributing nearly 0.
A good low-rank approximation algorithm achieves the guarantee \cref{eq:spectral} with rank $k = \order(\deff)$ or $k = \order(\deff \log \deff)$.
Our second main result shows \RPCholesky achieves such a guarantee.

\begin{theorem}[\RPCholesky: Spectral-norm error] \label{thm:spectral}
	Let $\mat{A}$ be psd, $\ridge > 0$ be a real number, and $\delta \in (0,1)$.
	Then the output of \RPCholesky is a $\ridge$-spectral approximation to $\mat{A}$ with probability at least $1-\delta$ after $k \ge 30\deff \log(16\deff / \delta)$ steps.
\end{theorem}

I have chosen not to optimize the constants in this analysis to make the argument cleaner.
The $k \sim \deff \log \deff$ scaling is necessary, not an artifact of the proof.
Here is an informal argument.
Consider a diagonal matrix $\mat{A} = 2\mu \cdot \diag(1,\ldots,1,r/(N-r),\ldots,r/(N-r))$, where ``$1$'' is repeated $r$ times.
Consider the limiting behavior of \RPCholesky on this example when $N\uparrow \infty$.
The effective dimension of this matrix is $\deff \sim (8/3)r$, and all of the ``large'' diagonal entries $2\mu$ must be selected by the algorithm to achieve the spectral guarantee \cref{eq:spectral}.
At the beginning of the algorithm, each step has a $\sim r/(2r) = 1/2$ probability of selecting a large entry, so it takes $\sim 2$ iterations in expectation to select the first large diagonal entry.
After the first large entry is chosen, then the probability of choosing a large entry is $\sim(r-1)/(2r-1)$, so an expected $\sim (2r-1)/(r-1)$ steps are needed to select the next large entry.
Continuing in this way, we see that an expected $\sim 2r/r + (2r-1)/(r-1) + \cdots + (r+1)/1 = \Theta(r\log r) = \Theta(\deff \log \deff)$ steps are needed to select all large diagonal entries.
This argument may be formalized and strengthened to show that \RPCholesky requires $\Omega(\deff \log \deff)$ steps to achieve the guarantee \cref{eq:spectral} with any constant success probability; I omit the details.

\subsection{Discussion: Spectral-norm error}

The spectral-norm bound \cref{eq:spectral} has a number of applications; see \cite[App.~E]{MM17} for several examples of how the guarantee \cref{eq:spectral} can be used to analyze low-rank approximation-accelerated versions of various kernel machine learning algorithms.

Another application is preconditioning \cite{DEF+23}.
Suppose we wish to solve the linear system $(\mat{A} + \ridge \Id) \vec{x} = \vec{b}$.
To do so, we may apply conjugate gradient with the Nystr\"om preconditioner $\Ahat + \ridge \Id$ using the approximation $\Ahat$ generated by \RPCholesky.
In experiments reported in \cite{DEF+23,ETW25a}, this strategy was shown to be among the best-available preconditioners for such regularized psd linear systems.

How many steps do we need to run \RPCholesky to guarantee the quality of this approach?
Held back by the lack of a spectral-norm bound, the existing analysis \cite[Thm.~1]{DEF+23} only certifies that \RPCholesky controls the preconditioner when run for $\dtail(1 + \log(\tr\mat{A} / \ridge))$ steps, where the tail rank is $\dtail \coloneqq \min \{ r : \sum_{i > r} \lambda_i(\mat{A}) \le \ridge \}$.
This parameter can be much larger than the effective dimension when the eigenvalues decay at a slow polynomial rate.
Indeed:
\begin{equation*}
	\text{When } \lambda_i(\mat{A}) = i^{-2} \text{ for } i =1,\ldots,N, \quad \text{it holds that } \dtail = \Theta(\ridge^{-1}) \text{ and } \deff = \Theta(\ridge^{-1/2}).
\end{equation*}
The tail rank is quadratically larger than the effective dimension in this case.
By using \cref{thm:spectral}, we see that running for $\order(\deff \log \deff)$ steps suffices to control the condition number, significantly improving the analysis of \RPCholesky preconditioning in the slow eigenvalue decay regime.

\begin{corollary}[\RPCholesky preconditioning] \label{cor:preconditioning}
	Let $\ridge > 0$ and $\delta \in (0,1)$, and let $\Ahat$ denote the approximation produced by running \RPCholesky on $\mat{A}$ for $k\ge 30\deff \log(16\deff/\delta)$ steps.
	Then, with probability at least $1-\delta$, the \RPCholesky preconditioner $\Ahat + \ridge \Id$ controls the condition number:
	\begin{equation*}
		\operatorname{cond}\big((\Ahat + \ridge \Id)^{-1/2}(\mat{A} + \ridge \Id)(\Ahat + \ridge \Id)^{-1/2}\big) \le 2 \quad \text{where } \operatorname{cond}(\mat{M}) \coloneqq \lambda_{\max }(\mat{M}) / \lambda_{\rm min}(\mat{M}).
	\end{equation*}
\end{corollary}

The constants in \cref{cor:preconditioning} are larger than those in the known result \cite[Thm.~1]{DEF+23}, so this new bound is not always quantitatively sharper than the existing result.
Nonetheless, it is comforting to know that \RPCholesky preconditioning requires only $\order(\deff \log \deff)$ steps to guarantee control on the condition number, matching peer alternative preconditioning strategies \cite{MM17,FTU23} up to logarithmic factors.

\begin{proof}[Proof of \cref{cor:preconditioning}]
	By \cref{thm:spectral}, \RPCholesky achieves the guarantee $\Ahat\preceq \mat{A} \preceq \Ahat + \ridge \Id$ after the stated number of steps.
	Therefore, $\Ahat + \mu \Id \preceq \mat{A} + \mu \Id \preceq \Ahat + 2\mu \Id \preceq 2 (\Ahat + \mu \Id)$.
	The stated guarantee follows by conjugating these relations by $(\Ahat + \mu \Id)^{-1/2}$.
\end{proof}

\subsection{Extensions}

Finally, we note a few immediate extensions.
The \RPCholesky method is algebraically equivalent to the earlier randomly pivoted QR (adaptive sampling) algorithm \cite{DRVW06,DV06,CETW25,Epp25a}; see \cite[secs.~2.6 and 9.2]{Epp25a}.
Using this connection, \cref{thm:main,cor:complexity} yield immediate analogs for randomly pivoted QR.
The techniques here also should immediately generalize to the continuous setting \cite{EM23a,Epp25a}, giving improved bounds for low-rank approximation of positive definite kernel functions with applications to kernel-based active learning methods and quadrature on reproducing kernel Hilbert spaces.

\subsection{AI use}

\Cref{known:rpcholesky} was the result of months of focused human attention \cite{CETW25}. 
After the release of the original manuscript, I continued studying the \RPCholesky algorithm with collaborators for the remainder of my PhD \cite{EM23a,ETW25a,Epp25a}.
Proving a result like \cref{cor:complexity} was my research ``white whale'' for many years.
I know other researchers searched for such a result as well.

Given this amount of attention, I found it startling when GPT 5.6-Sol (Pro) produced \cref{thm:main} and a version of \cref{cor:complexity} in just over two hours.
The original argument supplied an $\order(r/\varepsilon + r\log r)$ complexity, and the author discovered the easy improvement to a complexity of $\order(r/\varepsilon + r\sqrt{\log r})$.
\Cref{thm:spectral} was, in essence, also produced from a single prompt.
ChatGPT Codex was used to help design and run numerical experiments, and AI tools were used to help in proofreading and literature search.
The manuscript was written by the author, who assumes full responsibility for the manuscript.

Now is truly a disorienting time for researchers in the mathematical sciences.
Just months ago, the state of the art largely consisted of human experts working interactively with AI systems, combining the skill and knowledge of humans with the explorative power of AI models; see \cite{Col26,DD26} for two examples in computational mathematics.
Now, AI systems are powerful enough that they can solve nontrivial research questions---a favorite research question worked on for many years, no less---in a single prompt.

\subsection{Outline}
The rest of this paper proves these results, with the proof of \cref{thm:main} in \cref{sec:proof}, the proof of \cref{cor:complexity} in \cref{sec:corollary}, and the proof of \cref{thm:spectral} in \cref{sec:spectral}.

\section{Proof of \cref{thm:main}}
\label{sec:proof}

The proof is evocative of Deshpande and Vempala's analysis of the randomly pivoted QR (adaptive sampling) method \cite[Prop.~1]{DV06} by bounding the probability of selecting any particular pivot sequence.
Let $\set{s} \coloneqq (s_1,\ldots,s_k)$ denote the random (ordered) tuple of pivots selected by the algorithm.
Throughout, we shall use $\set{t} = (t_1,\ldots,t_k)$ to denote some tuple of indices.
For any $\set{t}$, we let $\Ahat^{(q)}_{\set{t}}$ be the approximation produced by $q$ steps of the update rule \cref{eq:cholesky-update} with the pivots $\set{t}$, and we introduce the sequence of trace errors:
\begin{equation*}
	\varrho_q(\set{t}) \coloneqq \tr\bigl(\mat{A} - \Ahat^{(q)}_{\set{t}}\bigr) \quad \text{for } q=1,\ldots,k.
\end{equation*}
In particular, $\varrho_q(\set{s})$ are the random sequence of trace errors for the \RPCholesky algorithm.
Also, let
\begin{equation*}
	\tau_q \coloneqq \tr\big(\mat{A} - \lowrank{\mat{A}}_q\big)
\end{equation*}
denote the error of the best rank-$q$ approximation.
By optimality, it holds that
\begin{equation} \label{eq:eckart-young}
	\varrho_q(\set{t}) \ge \tau_q \quad \text{for any tuple } \set{t}.
\end{equation}
Our goal is to bound the probability of the event $\varrho_k(\set{s}) \ge (1+\varepsilon) \tau_r$.

The probability that the pivot sequence $\set{s}$ selected by \RPCholesky is any given sequence $\set{t}$ is given by the expression
\begin{equation*}
	\prob \{ \set{s} = \set{t} \} = \frac{\mat{A}(t_1,t_1) - \Ahat^{(0)}_{\set{t}}(t_1,t_1)}{\varrho_0(\set{t})} \times \cdots \times \frac{\mat{A}(t_k,t_k) - \Ahat^{(k-1)}_{\set{t}}(t_k,t_k)}{\varrho_{k-1}(\set{t})}.
\end{equation*}
A standard connection between the Cholesky decomposition and determinants shows that the product of the numerators is
\begin{equation} \label{eq:numerator}
	[\mat{A}(t_1,t_1) - \Ahat^{(0)}_{\set{t}}(t_1,t_1)] \cdots [\mat{A}(t_k,t_k) - \Ahat^{(k-1)}_{\set{t}}(t_k,t_k)] = \det \mat{A}(\set{t},\set{t});
\end{equation}
see \cite[Prop.~3.20]{Epp25a}.
We do not have such an exact formula for the denominator, but we can bound it.
When the error at step $k$ is large $\varrho_k(\set{t}) \ge (1+\varepsilon) \tau_r$, the error on all previous steps $q$ is large as well: $\varrho_q(\set{t}) \ge (1+\varepsilon) \tau_r$.
Combining this with \cref{eq:eckart-young} yields the bound
\begin{equation} \label{eq:denominator}
	\varrho_0(\set{t})\cdots \varrho_{k-1}(\set{t}) \ge (1+\varepsilon)^{k-r} \cdot \tau_0^{\vphantom{k-r}}\cdots \tau_{r-1}^{\vphantom{k-r}}\tau_r^{k-r} \quad \text{when } \varrho_k(\set{t}) \ge (1+\varepsilon) \tau_r.
\end{equation}
Thus, summing over all index tuples $\set{t}$ that yield a large error $\varrho_k(\set{t}) \ge (1+\varepsilon) \tau_r$ and combining the numerator expression \cref{eq:numerator} with the denominator bound \cref{eq:denominator} gives
\begin{equation} \label{eq:prob-big-err}
	\prob \{ \varrho_k(\set{t}) \ge (1+\varepsilon) \tau_r \} \le \frac{\sum_{\set{t}} \det \mat{A}(\set{t},\set{t})}{(1+\varepsilon)^{k-r} \cdot \tau_0^{\vphantom{k-r}}\cdots \tau_{r-1}^{\vphantom{k-r}}\tau_r^{k-r}} = \frac{k! \, \mathrm{e}_k(\vec{\lambda})}{(1+\varepsilon)^{k-r} \cdot \tau_0^{\vphantom{k-r}}\cdots \tau_{r-1}^{\vphantom{k-r}}\tau_r^{k-r}}.
\end{equation}
In the equality statement, we use the fact that the sum over all determinants of $k\times k$ principal submatrices is the $k$th elementary symmetric polynomial
\begin{equation*}
	\mathrm{e}_k(\vec{\lambda}) = \sum_{i_1 < \cdots < i_k} \lambda_{i_1} \cdots \lambda_{i_k}
\end{equation*}
evaluated at the eigenvalues $\vec{\lambda}$ of $\mat{A}$; see\cite[Lem.~2.1]{GS12}.
Each submatrix is indexed by $k!$ different \emph{ordered} tuples, yielding an additional $k!$ factor in the numerator.

We now bound $\mathrm{e}_k(\vec{\lambda})$ using techniques similar to those employed in the volume sampling literature \cite{DRVW06,GS12,BBC20,HH20a}.
We begin with the bound
\begin{equation*}
	\mathrm{e}_k(\vec{\lambda}) \le \sum_{\substack{i_1 < \cdots < i_r \\ r+1 \le i_{r+1} < \cdots < i_k}} \lambda_{i_1} \cdots \lambda_{i_k} = \mathrm{e}_r(\vec{\lambda}) \mathrm{e}_{k-r}(\lambda_{r+1},\ldots,\lambda_N),
\end{equation*}
The first factor is bounded by 
\begin{equation*}
	\tau_0 \tau_1 \cdots \tau_{r-1} = (\lambda_1 + \cdots + \lambda_n) (\lambda_2 + \cdots + \lambda_n) \cdots (\lambda_r + \cdots + \lambda_n)
\end{equation*}
This is easily seen by expanding this product as a sum and noticing all $r$-tuples of eigenvalues appear at least once.
To bound the second factor, observe that every product of the eigenvalues $\lambda_{r+1},\ldots,\lambda_N$ occurs $(k-r)!$ times when
\begin{equation*}
	\tau_r^{k-r} = (\lambda_{r+1} + \cdots + \lambda_N)^{k-r}
\end{equation*}
is expanded.
Combining these insights, we deduce
\begin{equation*}
	\mathrm{e}_k(\vec{\lambda}) \le \mathrm{e}_r(\vec{\lambda}) \mathrm{e}_{k-r}(\lambda_{r+1},\ldots,\lambda_N)\le (\tau_0 \cdots \tau_{r-1}) \times \frac{\tau_r^{k-r}}{(k-r)!}.
\end{equation*}
Substituting into \cref{eq:prob-big-err} proves the theorem.

\section{Proof of \cref{cor:complexity}}
\label{sec:corollary}

First, we develop the high-probability bound \cref{eq:whp-result}.
Weakening \cref{thm:main} slightly, it suffices to choose $k$ such that
\begin{equation*}
	\frac{k^r}{(1+\varepsilon)^{k-r}} \le \delta \iff -\frac{\log(1+\varepsilon)k}{r} \e^{-\log(1+\varepsilon)k/r} \ge -\frac{\log(1+\varepsilon)\delta^{1/r}}{r(1+\varepsilon)}.
\end{equation*}
The solution to this inequality is
\begin{equation*}
	-\frac{\log(1+\varepsilon)k}{r} \le \mathrm{W}_{-1}\left( -\frac{\log(1+\varepsilon)\delta^{1/r}}{r(1+\varepsilon)} \right) \iff k \ge k_{\rm hp}(r,\varepsilon,\delta)
\end{equation*}
for
\begin{equation} \label{eq:k_hp}
	k_{\rm hp}(r,\varepsilon,\delta)\coloneqq-\frac{r}{\log(1+\varepsilon)} \mathrm{W}_{-1} \left( -\frac{\log(1+\varepsilon)\delta^{1/r}}{r(1+\varepsilon)} \right).
\end{equation}
Here, $\mathrm{W}_{-1}$ is the lower real branch of the Lambert-W function \cite{CGH+96}.
It satisfies the asymptotics $-\mathrm{W}_{-1}(-u) \sim \log(1/u)$ as $u\downarrow 0$.
Combining with the standard relation $\log(1+\varepsilon) \sim \varepsilon$, we conclude that
\begin{equation*}
	k_{\rm hp}(r,\varepsilon,\delta) \sim \frac{r}{\varepsilon} \log \left(\frac{r}{\varepsilon\delta^{1/r}} \right) = \frac{r}{\varepsilon} \log\left( \frac{r}{\varepsilon} \right) + \frac{1}{\varepsilon} \log \left( \frac{1}{\delta} \right) \quad \text{as } r\uparrow\infty \text{ and } \varepsilon,\delta \downarrow 0.
\end{equation*}
To obtain a bound, weaken \cite{Cha13} to obtain the numerical inequality $-\mathrm{W}_{-1}(-u) \le 2\log(1/u)$ for $u \le 1/\mathrm{e}$ and use that numerical inequalities $\varepsilon/2 \le \log(1+\varepsilon) \le \varepsilon$ for $\varepsilon \in (0,1)$ to obtain
\begin{equation*}
	k_{\rm hp}(r,\varepsilon,\delta) \le \frac{4r}{\varepsilon} \log \left( \frac{4r}{\varepsilon \delta^{1/r}} \right) = \frac{4r}{\varepsilon} \log\left( \frac{4r}{\varepsilon} \right) + \frac{4}{\varepsilon} \log \left( \frac{1}{\delta} \right).
\end{equation*}

Next, we treat the expectation bound \cref{eq:expectation-result}.
For any nonnegative random variable $x$, we have the identity $\expect[x] = \int_0^\infty \prob \{ x \ge u \} \, \d u$.
Using this identity and decomposing the integral into $u < u_\star$ and $u > u_\star$ for $u_\star \coloneqq [k!/(k-r)!]^{1/(k-r)} \tr(\mat{A} - \lowrank{\mat{A}}_r)$ and applying \cref{thm:main}, we obtain
\begin{align*}
	\expect\bigl[ \tr\bigl(\mat{A} - \Ahat^{(k)}\bigr) \bigr]
	&= \int_0^{u_\star} \prob \{ \tr(\mat{A} - \Ahat^{(k)}) \ge u \}\,\d u + \int_{u_\star}^\infty \prob \{ \tr(\mat{A} - \Ahat^{(k)}) \ge u \} \, \d u \\
	&\le u_\star + \frac{k!}{(k-r)!} \tr(\mat{A} - \lowrank{\mat{A}}_r)^{k-r} \int_{u_\star}^\infty u^{-(k-r)} \,\d u = \frac{k-r}{k-r-1}\left( \frac{k!}{(k-r)!} \right)^{\tfrac{1}{k-r}} \tr(\mat{A} - \lowrank{\mat{A}}_r).
\end{align*}
Invoking this bound at $k = r + \ell$ for some $\ell \ge 2$ and bounding $k!/(k-r)! \le k^r$ yields
\begin{equation*}
	\expect\bigl[ \tr\bigl(\mat{A} - \Ahat^{(r+\ell)}\bigr) \bigr] \le 2\left( r + \ell \right)^{\tfrac{r}{\ell}} \tr(\mat{A} - \lowrank{\mat{A}}_r).
\end{equation*}

To make further progress, we import the following result:
\begin{known}[\RPCholesky: Contraction bound]
	Let $\mat{A}$ be a psd matrix, let $r,k_1,k_2\ge 1$ be integers, let $\varepsilon \in (0,1)$, and let $\Ahat^{(t)}$ be the result of $t$ steps of \RPCholesky on $\mat{A}$ for each $t$.
	Then 
	\begin{equation*}
		\expect\bigl[\tr\bigl(\mat{A} - \Ahat^{(k_1+k_2)}\bigr)\bigr] \le (1+\varepsilon) \tr(\mat{A} - \lowrank{\mat{A}}_r) \quad \text{when } k_2\ge \frac{r}{\varepsilon} + r \log \left( \frac{\expect\tr\bigl(\mat{A} - \Ahat^{(k_1)}\bigr)}{\varepsilon \tr(\mat{A} - \lowrank{\mat{A}}_r)} \right).
	\end{equation*}
\end{known}
This result combines \cite[Lem.~5.4]{CETW25} and some consequences developed in the last paragraph of the proof of \cite[Thm.~5.1]{CETW25}.
Applying this result with $k = r+\ell$ shows that
\begin{equation*}
	\expect\bigl(\tr\bigl(\mat{A} - \Ahat^{(k)}\bigr)\bigr) \le (1+\varepsilon) \tr(\mat{A} - \lowrank{\mat{A}}_r)
\end{equation*}
when
\begin{align*}
	k \ge r + \ell + \frac{r}{\varepsilon} + r \log \left( \frac{\ell\left( r + \ell \right)^{\tfrac{r}{\ell}}}{(\ell-1)\varepsilon} \right)= \frac{r}{\varepsilon} + r \log \left(\frac{1}{\varepsilon} \right) + \left(1+\log \frac{\ell}{\ell-1}\right) r + \ell + \frac{r^2\log (r+\ell)}{\ell}.
\end{align*}
Assume $r\ge 6$, set $\ell \coloneqq \lceil r \sqrt{\log r} \rceil \le r\sqrt{\log r} + 1$, and simplify to obtain a threshold of 
\begin{equation*}
	k\ge \frac{r}{\varepsilon} + r\log \left(\frac{1}{\varepsilon}\right) + g(r) \quad \text{for } g(r) \coloneqq (1+\log 1.5)r + r \sqrt{\log r} + 1 + r \frac{\log(r + r \sqrt{\log r} + 1)}{\sqrt
		{\log r}}.
\end{equation*}
The function $h(r) \coloneqq \big(g(r) - 2r\sqrt{\log r}\,\big)/r$ is decreasing, so it satisfies $g(r) \le 2r\sqrt{\log r} + h(6)r < 2r\sqrt{\log r} + 2.3r$ for $r\ge 6$.
The stated complexity result also holds for $r\le5$ by appealing to \cref{known:rpcholesky}.

\section{Proof of \cref{thm:spectral}} \label{sec:spectral}

The basic idea of the proof will be as follows.
Introduce the ramp function $\psi(t) \coloneqq (t - \ridge/2)_+$, and consider the quantity $\tr \psi(\mat{A})$.
This quantity has two desirable properties.
First, it allows us to detect if $\mat{A}$ has an eigenvalue of size $\mu$ or larger.
Indeed, if $\mat{A}$ does have such an eigenvalue, then $\tr \psi(\mat{A}) \ge \ridge/2$.
Second, $\tr \psi(\mat{A})$ is insensitive to the number and size of any eigenvalue of $\mat{A}$ smaller than $\ridge/2$.
Taken together, these properties make $\tr \psi(\mat{A})$ an excellent tool to bound the spectral-norm error of $\mat{A} - \Ahat$, as it detects large eigenvalues but is insensitive to the number and size of sufficiently small eigenvalues.

Our first main lemma shows that this quantity of interest decreases exponentially as a function of the number of \RPCholesky steps.

\begin{lemma}[Ramp function of the many-step residual] \label{lem:ramp}
	Let $\mat{A}$ be a psd matrix, let $\mat{A}^{(k)} \coloneqq \mat{A} - \Ahat^{(k)}$ denote the $k$-step \RPCholesky residual, and introduce the ramp function $\psi(t) \coloneqq (t - \ridge/2)_+$.
	Then
	\begin{equation*}
		\expect \tr \big(\psi(\mat{A}^{(k)})\big) \le \exp\left(-\frac{\ridge}{2\tr \mat{A}}\cdot k\right) \tr \mat{A}.
	\end{equation*}
\end{lemma}

To prove this lemma, we develop a single-step bound for the expected trace of any convex function of the \RPCholesky residual.
This result can be seen as a generalization of the exact formula for the trace error of \RPCholesky \cite[sec.~5.1]{CETW25} to higher powers.

\begin{lemma}[Convex function of the single-step residual] \label{lem:convex}
	Let $\mat{A}$ be a psd matrix, let $\mat{A}'$ be the residual after applying one step of \RPCholesky, and let $f : \real_+ \to \real_+$ be a convex function satisfying $f(0) = 0$.
	Then
	\begin{equation*}
		\expect \tr f(\mat{A}') \le \tr f(\mat{A}) - \frac{\tr(\mat{A}f(\mat{A}))}{\tr \mat{A}}.
	\end{equation*}
\end{lemma}

\begin{proof}
	Let $s$ denote the randomly selected pivot.
	Our starting point is a standard tool in the theory of partial Cholesky (Nystr\"om) approximations, the projection formula \cite[p.~6]{Git11}:
	\begin{equation*}
		\mat{A}' = \mat{A}^{1/2}\mat{\Pi}\mat{A}^{1/2} \quad \text{for } \mat{\Pi} \coloneqq \Id - \vec{u}\vec{u}^* \text{ and } \vec{u} = \mat{A}^{1/2}(:,s)/\sqrt{\mat{A}(s,s)}.
	\end{equation*}
	The matrices $\mat{A}'$ and $\mat{\Pi}\mat{A}\mat{\Pi}$ have the same eigenvalues, so applying Jensen's trace inequality \cite{HP03} gives
	\begin{equation*}
		\tr f(\mat{A}') = \tr f(\mat{\Pi}\mat{A}\mat{\Pi}) \le \tr(\mat{\Pi}f(\mat{A})\mat{\Pi}) = \tr f(\mat{A}) - \frac{[\mat{A}f(\mat{A})](s,s)}{\mat{A}(s,s)}.
	\end{equation*}
	Here, we have used the fact that $f(\mat{A})$ and $\mat{A}^{1/2}$ commute.
	Taking the expectation over the randomly selected index $s$ gives the result.
\end{proof}

\begin{proof}[Proof of \cref{lem:ramp}]
	Apply \cref{lem:convex} conditionally to obtain
	\begin{equation*}
		\expect \big[ \tr \big(\psi(\mat{A}^{(k+1)}\big) \mid \mat{A}^{(k)}\big] \le \tr \psi(\mat{A}^{(k)}) - \frac{\tr(\mat{A}^{(k)}\psi(\mat{A}^{(k)}))}{\tr \mat{A}^{(k)}} \le \tr \psi(\mat{A}^{(k)}) - \frac{\ridge \tr(\psi(\mat{A}^{(k)}))}{2\tr \mat{A}}.
	\end{equation*}
	In the second bound, we note that the trace residual of \RPCholesky is decreasing and observe that, since $\psi(t)$ is nonzero only when $t\ge \ridge/2$, we have the bound $t \psi(t) \ge (\ridge/2)\psi(t)$.
	Taking the expectation of both sides yields the recurrence
	\begin{equation*}
		\expect \tr \big(\psi(\mat{A}^{(k+1)})\big) \le \expect \tr \psi(\mat{A}^{(k)}) - \frac{\ridge \expect \tr(\psi(\mat{A}^{(k)}))}{2 \tr \mat{A}}.
	\end{equation*}
	Solving this recurrence with the initial condition $\tr \psi(\mat{A}^{0}) = \tr \psi(\mat{A}) \le \tr\mat{A}$ and using the numeric inequality $1-u \le \e^{-u}$ yields the stated bound.
\end{proof}

\Cref{lem:ramp} shows that the quantity of interest $\tr \psi(\mat{A})$ decreases at an exponential rate depending on the ratio $\ridge / \tr \mat{A}$.
Unfortunately, the rate will be very slow if $\tr \mat{A} \gg \ridge$.
To address this, we can use a warm start argument to show that the trace-residual becomes small after a small burn-in period.

\begin{proposition}[Warm start] \label{prop:warm}
	Let $\mat{A}$ be a psd matrix, and let $\delta \in (0,1)$.
	After running for $$k\ge 4\lceil 2\deff\rceil \log (4\lceil 2\deff\rceil) + 4\log(1/\delta)$$ steps, \RPCholesky achieves the bound $\tr(\mat{A} - \Ahat^{(k)}) \le 4\ridge\deff$ with probability at least $1-\delta$.
\end{proposition}

\begin{proof}
	We begin with a standard algebraic manipulation involving the effective dimension.
	In the sum
	\begin{equation*}
		\deff = \frac{\lambda_i(\mat{A})}{\lambda_i(\mat{A}) + \ridge},
	\end{equation*}
	each summand with $\lambda_i(\mat{A}) \ge \ridge$ is at least $1/2$.
	Thus, there are at most $2\deff$ eigenvalues greater than $\ridge$, and thus, for any $i > 2\deff$, it holds that
	\begin{equation*}
		\lambda_i(\mat{A}) \le \ridge \le \frac{2\ridge \lambda_i(\mat{A})}{\lambda_i(\mat{A}) + \ridge}.
	\end{equation*}
	Thus, for $r\coloneqq \lceil 2\deff\rceil$, it holds that
	\begin{equation*}
		\tr(\mat{A} - \lowrank{\mat{A}}_r) = \sum_{i = r+1}^N \lambda_i(\mat{A}) \le \sum_{i=r+1}^N \frac{2\ridge \lambda_i(\mat{A})}{\lambda_i(\mat{A}) + \ridge} \le 2\ridge\, \deff.
	\end{equation*}
	By \cref{eq:whp-result} with $\varepsilon = 1$, the stated guarantee holds provided that $k \ge 2\lceil 2\deff\rceil \log (2\lceil 2\deff\rceil) + 4\log(1/\delta)$.
\end{proof}

Combining these ingredients, we prove the paper's second main result.

\begin{proof}[Proof of \cref{thm:spectral}]
	Introduce the ramp function $\psi(t) \coloneqq (t - \ridge/2)_+$.
	By the comments at the start of this section, it is sufficient to show that $\tr \psi(\mat{A} - \Ahat^{(k)}) \le \ridge/2$ with probability $1-\delta$.
	Set
	\begin{equation*}
		k_1 \coloneqq \lceil 4\lceil 2\deff\rceil \log (4\lceil 2\deff\rceil) + 4\log(2/\delta)\rceil \quad \text{and} \quad k_2 \coloneqq \lceil 8\deff \log(16 \deff/\delta)\rceil.
	\end{equation*}
	\Cref{prop:warm} tells us that, after $k_1$ steps, we obtain the bound $\tr(\mat{A} - \Ahat^{(k_1)}) \le 4\ridge\deff$ with probability at least $1-\delta/2$.
	Call this event $\mathcal{E}$.
	Applying \cref{lem:ramp} conditionally on $\mathcal{E}$, we see that
	\begin{equation*}
		\expect [\tr \psi(\mat{A} - \Ahat^{(k_1+k_2)}) \mid \mathcal{E}] \le 4\ridge\deff\exp\left(-\frac{k_2}{8\deff} \right)  \le \frac{\delta\mu}{4}
	\end{equation*}
	Thus, by Markov's inequality
	\begin{equation*}
		\prob \{ \tr \psi(\mat{A} - \Ahat^{(k_1+k_2)}) > \ridge/2\} \le \prob (\mathcal{E}^{\mathsf{c}}) + \prob \{ \tr \psi(\mat{A}^{(k_1+k_2)}) > \ridge/2 \mid \mathcal{E} \} \le  \frac{\delta}{2} + \frac{\delta}{2} = \delta.
	\end{equation*}
	
	Lastly, we determine the complexity.
	First, note that we are free to assume $\deff > 1/2$, since even the zero matrix is a $\ridge$-spectral approximation when $\deff \le 1/2$.
	From there, a tiny bit of case work proves $k_1 + k_2 \le 30\deff \log(16\deff/\delta)$, completing the argument.
\end{proof}


\subsection*{Acknowledgements}

The author thanks Chris Cama\~no, Aidan Epperly, Raphael Meyer, Yuji Nakatsukasa, Joel Tropp, and Robert Webber for valuable discussions.
This work was supported by the Miller Institute for Research in Basic Research in Science at the University of California, Berkeley, and the author used AI resources provided by Lawrence Berkeley National Lab.

\appendix


\scriptsize
\bibliographystyle{halpha}
\let\oldthebibliography=\thebibliography
\let\endoldthebibliography=\endthebibliography
\renewenvironment{thebibliography}[1]%
  {\begin{oldthebibliography}{#1}%
   \setlength{\itemsep}{2pt}%
   \setlength{\parskip}{2pt}%
  }%
  {\end{oldthebibliography}}
\bibliography{otherrefs}

\end{document}